\documentclass[11pt]{amsart}
\usepackage{amssymb,amsmath,amsfonts}
\usepackage{amsthm,amstext,array}
\usepackage{latexsym}
\usepackage{a4wide}
\usepackage{epsfig}
\usepackage{bbm}
\usepackage{mathtools}
\usepackage{tikz}
\usepackage{enumerate}

\newcommand{\Z}{\ensuremath{\mathbb Z}}
\newcommand{\R}{\ensuremath{\mathbb R}}

\renewcommand{\rho}{\varrho}
\renewcommand{\phi}{\varphi}
\newcommand{\dlim}{\displaystyle\lim} 
\newcommand{\dsum}{\displaystyle\sum} 
\newcommand{\vect}[2]{\left(\protect\begin{smallmatrix} #1 \\ #2 \end{smallmatrix}\protect\right)}

\DeclareMathOperator{\cl}{cl}
\DeclareMathOperator{\diam}{diam}

\DeclareMathOperator{\cone}{cone}
\DeclareMathOperator{\sym}{sym}
\DeclareMathOperator{\intr}{int}

\def\set#1{\left\{#1\right\}} 

\newtheorem{thm}{Theorem}[section]

\newtheorem{prop}[thm]{Proposition}

\begin{document}

\title[Highly symmetric fundamental domains]
{Highly symmetric fundamental domains for lattices in $\R^2$ and $\R^3$}  

\author{Joseph Ray Clarence G.~Damasco}
\address{Institute of Mathematics, College of Science, University of the Philippines Diliman, Quezon City 1101, Philippines}
\email{jrcgdamasco@math.upd.edu.ph}

\author{D. Frettl\"oh}
\address{Technische Fakult\"at, Universit\"at Bielefeld /
Fachbereich Mathematik und Informatik, FU Berlin}
\email{dfrettloeh@techfak.uni-bielefeld.de}
\urladdr{http://math.uni-bielefeld.de/~frettloe}

\author{Manuel Joseph C.~Loquias}
\address{Institute of Mathematics, College of Science, University of the Philippines Diliman, Quezon City 1101, Philippines}
\email{mjcloquias@math.upd.edu.ph}

\begin{abstract} 
It is shown that most lattices $\Gamma$ in $\R^2$ and $\R^3$ possess 
a fundamental domain $F$ for the action of $\Gamma$ on $\R^2$, 
respectively $\R^3$, having more symmetries than the point group 
$P(\Gamma)$, i.e., the group $P(\Gamma) \subset O(d)$ fixing  
$\Gamma$. In particular, $P(\Gamma)$ is a subgroup 
of the symmetry group $S(F)$ of $F$ of index 2 in these cases. 
Exceptions are cubic lattices in the three-dimensional case,
where such an $F$ does not exist.
Possible exceptions are rhombic lattices in the plane case,
where the constructions presented here do not seem to work.
\end{abstract} 

\maketitle

\section{\bf Introduction} \label{sec:intro}
The question inspiring this work is ``Given a lattice $\Gamma$ 
in $\R^d$, how many symmetries can a fundamental domain for the
action of $\Gamma$ on $\R^d$ have?''. We will address this question
for lattices in $\R^2$ and $\R^3$ and provide some answers.

A {\em lattice} in $\R^d$ is the $\Z$-span of $d$ linearly independent
vectors in $\R^d$. The point group $P(\Gamma)$ of a lattice $\Gamma$ 
in $\R^d$ is the set of Euclidean motions fixing both $\Gamma$ and 
the origin. In other words, $P(\Gamma) \subset O(d)$ is the set of
orthogonal maps fixing $\Gamma$. It is clear that each lattice 
$\Gamma$ has a fundamental domain having $P(\Gamma)$ as its symmetry group,
see Proposition \ref{prop:sfpg}. (For more detailed definitions 
see below.) For instance, consider the square lattice $\Z^2$ in the plane
$\R^2$. Its point group $P(\Z^2)$ is the dihedral group $D_4$ of order 
eight, containing rotations by $0, \, \pi/2, \, \pi, \, 3\pi/2$,
together with four reflections. One possible fundamental domain 
of $\Z^2$ is a unit square, centered at the origin. Clearly
the symmetry group of this unit square is $D_{4}$ as well.

In this paper we show that most lattices in $\R^2$ and $\R^3$
possess fundamental domains with more symmetry than the point group
of the lattice. In general, these
fundamental domains will be neither simply connected, nor will their
interiors be connected. Some of these domains are of fractal appearance. 
The two main results are the following. 

\begin{thm} \label{thm:satz1}
Let $\Gamma \subset \R^2$ be a lattice with point group 
$P(\Gamma)$, such that $\Gamma$ is not a rhombic lattice. 
Then there is a compact fundamental domain $F$ of \ $\Gamma$ with
symmetry group $S(F)$ such that $P(\Gamma)$ is a subgroup of $S(F)$
of index $[S(F):P(\Gamma)]=2$.
\end{thm}

\begin{thm} \label{thm:satz2}
Let $\Gamma \subset \R^3$ be a lattice with point group 
$P(\Gamma)$, such that $\Gamma$ is not a cubic lattice. 
Then there is a compact fundamental domain $F$ of \ $\Gamma$ with 
symmetry group $S(F)$ such that $P(\Gamma)$ is a subgroup of $S(F)$
of index $[S(F):P(\Gamma)]=2$.
\end{thm}

In the remainder of this section the necessary definitions and 
notations are introduced. Section \ref{sec:d2} 
is dedicated to the proof of Theorem \ref{thm:satz1},
Section \ref{sec:d3} contains the proof of Theorem \ref{thm:satz2}.
Section 4 contains some remarks and further questions.

{\bf Notation:} We denote the cyclic group of order $n$ by $C_n$, and 
the dihedral group of order $2n$ by $D_n$.  The orthogonal
group over $\R^d$ is denoted by $O(d)$. This group can be identified 
with the group of Euclideanmotions (i.e. isometries of $\R^d$, 
including reflections) fixing the origin. The closure of a
set $A \subset \R^d$ is denoted by $\cl(A)$. For any set 
$X \subset \R^d$, let $S(X)$ denote the symmetry group of $X$, that
is, the set of all Euclidean motions (including reflections and
translations) $\varphi: \R^d \to \R^d$ with
$\varphi(X)=X$.  The {\em cone} centered at $x$ spanned by $m$ vectors
$v_1, \ldots v_m \in \R^d$ is defined by 
\[ \cone(x; v_1, \ldots , v_m) = 
\{  x+ \sum\limits_{i=1}^m \lambda_i v_i : \lambda_i \ge 0 \}.\] 
A sum $A+B$, where $A,B \subset \R^d$,  always means the 
{\em  Minkowski sum}  
\[ A+B = \{ a+  b \, | \, a \in A, b \in B \}. \]  
The line segment with endpoints 
$x, y \in \R^d$ is denoted by $[x, y]$, while $[x, y] \setminus \{x,y\}$ is denoted $(x,y)$.  
A {\em lattice} in $\R^d$ is a discrete cocompact subgroup of
$\R^d$. Any lattice in $\R^d$ can be written as $\langle b_1,
\ldots, b_d \rangle_{\Z}$, where $b_1, \ldots, b_d$
span $\R^d$. Such a set $\{b_1, \ldots, b_d\}$ is called a 
{\em  basis} of the lattice. A basis of a given lattice is not unique. 

A fundamental domain for the action of a lattice $\Gamma$ on $\R^d$
is a set $F$ such that $F$ 
contains exactly one representative for each element of $\R^d / \Gamma$. 
Here we want to consider nice geometric representations of fundamental
domains. In particular we want to consider compact sets. Hence we allow 
a fundamental domain $F$ to contain more than one representative for 
some element in $\R^d / \Gamma$ if these representatives are all
contained in the boundary of $F$. 
For instance, a proper fundamental domain of the action of $\Z^d$ 
on $\R^d$ is the $d$-torus, and a fundamental domain in the sense of 
this paper of the action of the lattice $\Z^d$ on $\R^d$ is the 
$d$-dimensional unit cube $[0,1]^d$. 
A particular fundamental domain of a lattice $\Gamma = \langle 
b_1, \ldots, b_d \rangle_{\Z}$ is the {\em fundamental
parallelepiped} $[0, b_1]  +  \ldots + [0, b_d]$. 
Note that for any fundamental domain $F$ of a lattice
$\Gamma$, $\{ F + g \, | \, g \in \Gamma \}$ is a tiling of
$\R^d$. A {\em tiling} of $\R^d$ is a packing of $\R^d$ which
is also a covering of $\R^d$. In other words, a tiling is a 
covering of $\R^d$ by pairwise non-overlapping compact sets $T_i$.
Two compact sets are {\em non-overlapping} if their interiors 
are disjoint. 

Trivially, the symmetry group $S(\Gamma)$ of any lattice contains a
subgroup isomorphic to $\Gamma$, namely, the group of all 
translations by elements of $\Gamma$. 
The subgroup $P(\Gamma) = S(\Gamma) / \Gamma$ is called 
{\em point group} of $\Gamma$. For lattices in $\R^d$, one has:
\[ S(\Gamma) = P(\Gamma) \ltimes \Gamma. \]
The following fact is usually called the crystallographic restriction
(see for instance \cite{cox}, Section 4.5). 

\begin{prop} \label{prop:crystrest}
Rotations fixing a lattice in $\R^2$ or $\R^3$ are either 2-fold, 3-fold,
4-fold or 6-fold.
\end{prop}

The {\em Voronoi cell} of a lattice point $x$ in $\R^d$ is the set of 
points in $\R^d$ whose distance to $x$ is not greater than their 
distance to any other lattice point. It is easy to see that 
for any lattice $\Gamma \subset \R^d$ the closed Voronoi
cell $V=V(0)$ of $0$ is a fundamental domain of $\Gamma$ with 
$P(\Gamma) \subseteq S(V)$. Now, let $\sigma \in S(V)$. For $v \in \Gamma$, 
let $\pi_{v}$ be the perpendicular bisector of $[0,v]$, and $H_{v}$ be the 
half-space delineated by $\pi_{v}$ containing $0$. Then, 
$V=\displaystyle \bigcap_{v \in \Gamma \setminus \{0\}} H_{v}$. 
In particular, there exist $v_{1}$, $\ldots$, $v_{n}$ such that  
$V=\displaystyle \bigcap_{i=1}^{n} H_{v_{i}}$, where each $H_{v_{i}}$ is 
called a supporting hyperplane of $V$. 

We claim that $\langle v_{1}, \ldots, v_{n} \rangle_{\Z} = \Gamma$. First, we 
note that $\Gamma' = \langle v_{1}, \ldots, v_{n} \rangle_{\Z}$ is a lattice 
of full rank $d$, otherwise  $v_{1}$, $\ldots$, $v_{n}$ are all contained in 
one hyperplane in $\R^{d}$ and $\displaystyle \bigcap_{i=1}^{n} H_{v_{i}}$ 
cannot be a fundamental domain for $\Gamma$. Moreover, $\Gamma'$ cannot be a 
proper sublattice of $\Gamma$, because the Voronoi cell of $0$ in $\Gamma'$ must also 
be $\displaystyle \bigcap_{i=1}^{n} H_{v_{i}}$.

For $i=1$, $\ldots$, $n$, let $f_{v_{i}}$ be the face of $V$ contained in $H_{v_{i}}$. 
Now, for each $i$, there exists $j \in \{1,\ldots,n\}$ such that 
$\sigma\big(f_{v_{i}}\big)=f_{v_{j}}$. It follows that 
$\sigma\big(\pi_{v_{i}}\big)=\pi_{v_{j}}$, 
and $\sigma\big(\frac{1}{2}v_{i}\big)=\frac{1}{2}v_{j}$, 
as $\frac{1}{2}v_{i}$ and $\frac{1}{2}v_{j}$ are the unique points on 
$\pi_{v_{i}}$ and $\pi_{v_{j}}$, respectively, of minimal distance to $0$. 
Hence, $\sigma$ permutes $\{v_{1},\ldots,v_{n}\}$ and fixes $0$, so that 
$S(V) \subseteq P(\Gamma)$. We thus have the following proposition. 

\begin{prop} \label{prop:sfpg}
If $\Gamma$ is a lattice in $\R^d$ then $\Gamma$ has a 
fundamental domain $F$ such that $S(F) = P(\Gamma)$. 
\end{prop}

We will use orbifold notation to denote planar symmetry groups in the 
sequel, compare \cite{cgb}. For instance, $\ast 442$ denotes the 
symmetry group $S(\Z^2)$ of the square lattice $\Z^2$, and $\ast432$
denotes the symmetry group of the cube. For a translation of orbifold
notation into your favourite notation, see \cite{cgb} or \cite{wik2}.
In principle we can denote cyclic groups $C_n$ and dihedral groups 
$D_n$ in orbifold notation, too. Since the symbol for $C_n$---regarded as
the symmetry group of some object in the plane---is just $n$ in
orbifold notation, we will rather use the former abbreviation for 
the sake of clarity.

\section{\bf Dimension 2} \label{sec:d2}

It is well known that each finite group of Euclidean motions in the 
plane is either $C_n$ or $D_n$. By the crystallographic restriction 
(Proposition \ref{prop:crystrest}) there are just 10 candidates 
for such groups being point groups of a planar lattice, namely
\[ C_1, C_2, C_3, C_4, C_6, D_1, D_2, D_3, D_4, D_6. \]
Note that $C_2$ and $D_1$ are equal as abstract groups, since there is 
only one group of order two up to isomorphisms. But since we are dealing 
with groups of Euclidean motions, we will use the convention that 
a cyclic group $C_n$ contains rotations only, and a dihedral group 
$D_n$ contains $n$ rotations (including the identity) and $n$ reflections.  
The fact that each planar lattice is fixed under a rotation through $\pi$
about the origin implies that $C_1, \, C_3, \, D_1$ and $D_3$ cannot be
point groups of any planar lattice. Some further thought yields the
following result.

\begin{prop} 
If $\Gamma$ is a lattice in $\R^2$, then $P(\Gamma) \in \{C_2, D_2,
D_4, D_6 \}$, and $S(\Gamma) \in \{\ast632$, $\ast442,$ $\ast2222,$ 
$2\ast22,$ $2222 \}$.
\end{prop}

This result is well known. Nevertheless, since we are not aware of a 
decent reference, we will sketch the proof here.

\begin{proof}
We consider the distinct possibilities of properties of basis vectors
of $\Gamma$. First, if $\Gamma$ has a basis of two orthogonal
vectors of equal length, this yields (up to similarity) the square lattice 
$\Z^2$, with point group $D_4$ and symmetry group $\ast442$. 
Second, if $\Gamma$ has a basis of two vectors of equal length 
with angle $\pi/3$, this yields (up to similarity) the hexagonal 
lattice $A_2 = \langle (1,0)^T, (\frac{1}{2}, \frac{\sqrt{3}}{2})^T
\rangle_{\Z}$, with point group $D_6$ and symmetry group $\ast632$. 
Third, if $\Gamma$ has a basis of two vectors of equal 
length, but neither with angle $\pi/3$ nor $\pi/2$ 
nor $2\pi/3$, then $\Gamma$ is called {\em rhombic lattice} 
and has point group $D_2$ and symmetry group $2\ast22$. 
A planar lattice which has orthogonal basis vectors of different 
length (but not of equal length) is called {\em rectangular lattice}. 
It has also point group $D_2$, but its
symmetry group is $\ast2222$. In particular, the entire 
symmetry group of a rhombic lattice is not isomorphic to the
entire symmetry group of a rectangular lattice, even though their
point groups agree. (Compare \cite{schw}, p 210.) All other lattices are called 
{\em oblique lattices} and have point group $C_2$, and symmetry group
$2222$.  \end{proof}

Regarding the five cases above, in connection with 
Proposition~\ref{prop:sfpg} one obtains that one possible
fundamental domain for the hexagonal (respectively square, rectangular, rhombic, 
oblique) lattice is a regular hexagon (respectively square, rectangle, hexagon 
with $D_2$ symmetry, hexagon with $C_2$ symmetry). 
We will prove Theorem \ref{thm:satz1} by considering four out of
these five cases. The first two cases---the square lattice and the hexagonal
lattice---are due to V. 
Elser \cite{elser}. To the knowledge of the authora his proof has not
been published anywhere, so we give a detailed proof here.

\begin{prop}[Elser] \label{eins}
The square lattice $\Z^2$ has a fundamental domain
$F_{\square}$ with $S(F_{\square}) = D_8$. 
\end{prop}
\begin{proof}
The point group of the square lattice $\Z^2$ is $D_4$. 
The claim is proved by constructing a fundamental domain $F_{\square}$
of $\Z^2$ with symmetry group $D_8$.

Let $\Gamma = \Z^{2}$. Consider a regular octagon of edge length $\ell=\sqrt{2}-1$ 
oriented such that it has edges parallel to the coordinate axes. Let $P_{1}$ be the 
packing of the plane by copies of the octagon, with every point in $\Gamma$ having 
one copy centered at it. See Figure~\ref{fig:OctPackings}. The packing looks like 
the Archimedean tiling $4.8^2$ by octagons and squares, where the  squares are the 
holes of the packing. By the choice of $\ell$ and the orientation of the octagons, 
a pair of octagons intersect if and only if their centers are one unit apart, and 
the two intersect only at a common edge. That is, the octagons are pairwise 
non-overlapping. For $n \geq 2$, let $P_{n}$ be the packing by octagons having the 
same orientation as those in $P_{1}$ and of edge length $\ell^{n}$, centered at 
each vertex of every octagon in $P_{n-1}$. Similarly, a pair of octagons in the 
$n^{\text{th}}$ step intersect if and only if their centers are consecutive vertices 
in some octagon in $P_{n-1}$,  and the two intersect only at a common edge. In 
Figure~\ref{fig:OctPackings}, we have $\displaystyle\bigcup_{k=1}^{n} P_{k}$ for 
$n \leq 4$. 

	\begin{figure}[ht]
	\begin{center}
	\includegraphics[scale=0.25]{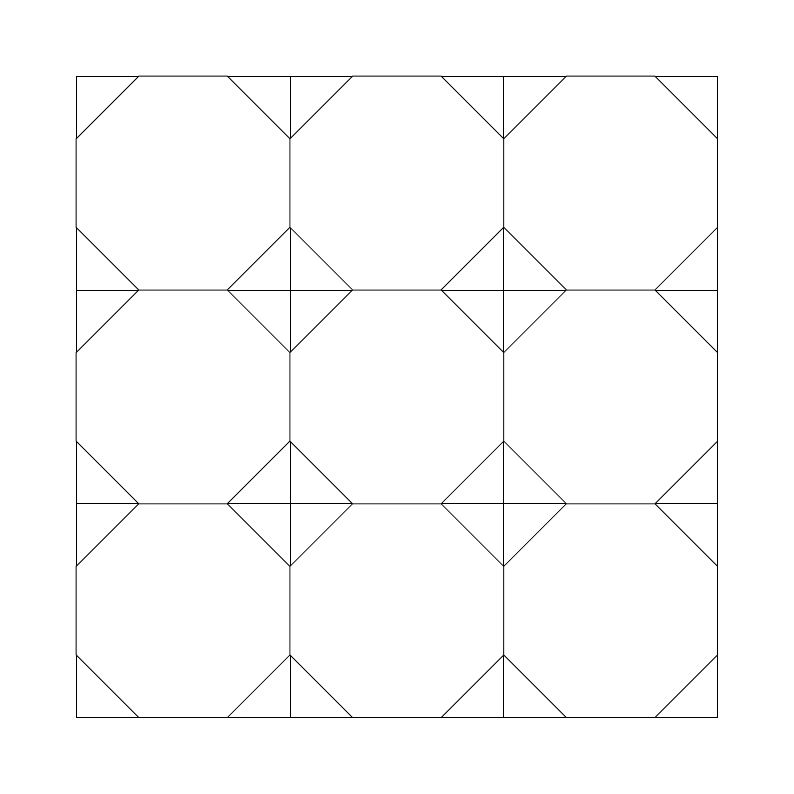}\quad
	\includegraphics[scale=0.25]{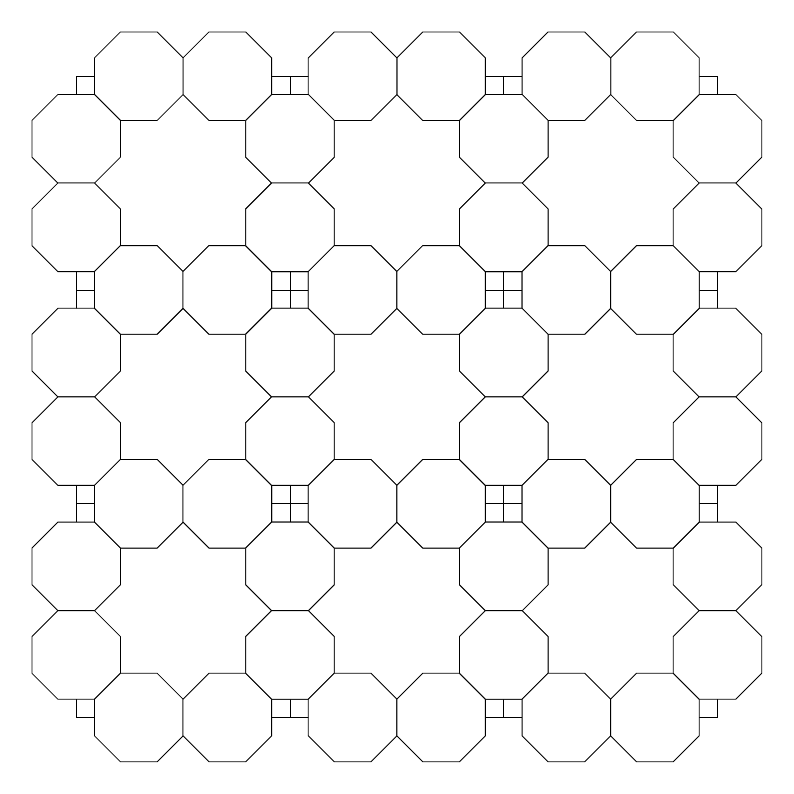}\quad
	\includegraphics[scale=0.25]{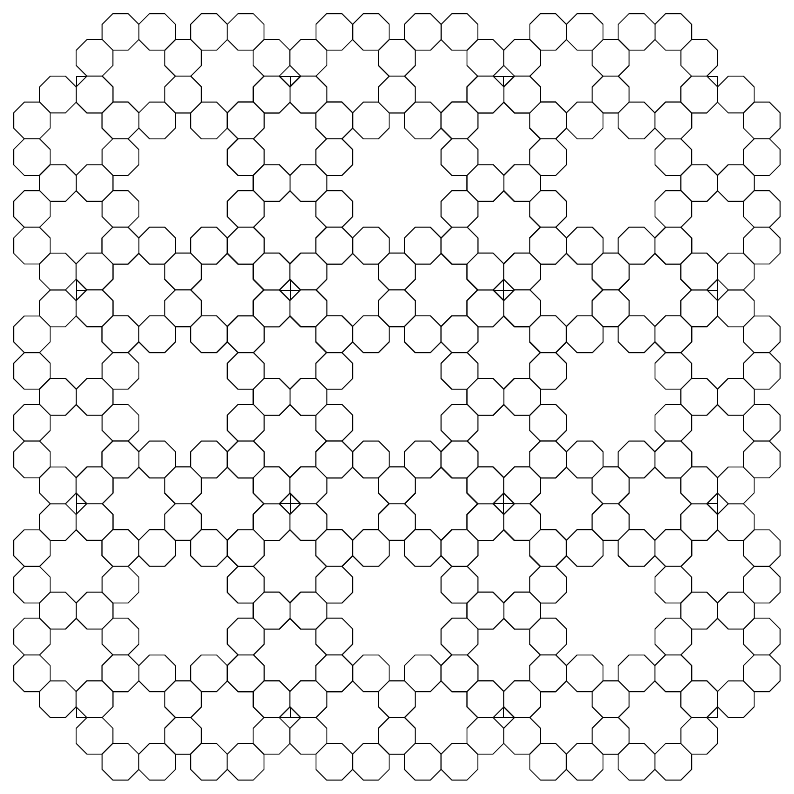}\quad
	\includegraphics[scale=0.25]{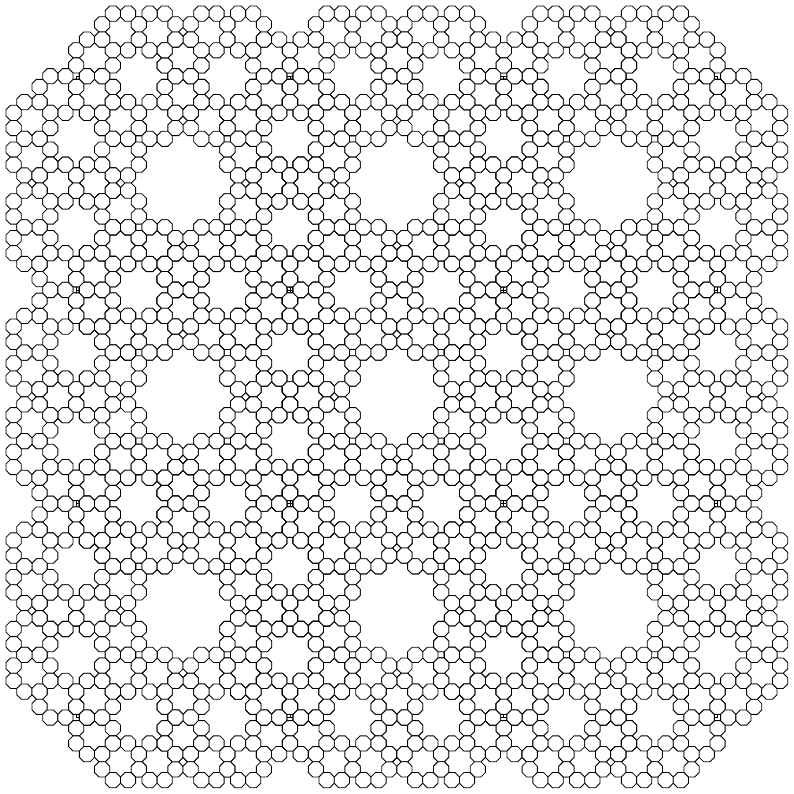}
	\caption{Filling the plane with octagons.}
	\label{fig:OctPackings}
	\end{center}
	\end{figure}

Define $P_{1}^{*}$ as the octagon in $P_{1}$ that is centered at the origin. 
For $n\geq 2$, let $P_{n}^{*}$ consist of the octagons of $P_{n}$ centered 
at the vertices of all octagons in $P_{n-1}^{*}$. See 
Figure~\ref{fig:OctCentralPackings} for illustrations of 
$\displaystyle\bigcup_{k=1}^{n} P_{k}^{*}$ for $n \leq 4$.

	\begin{figure}[ht]
	\begin{center}
	\includegraphics[scale=0.4]{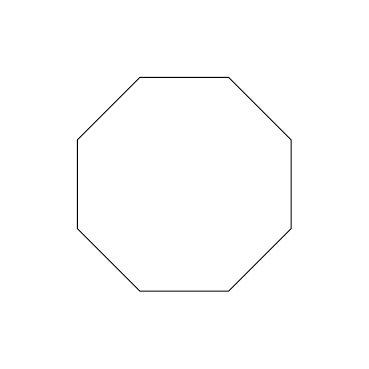}\quad
	\includegraphics[scale=0.4]{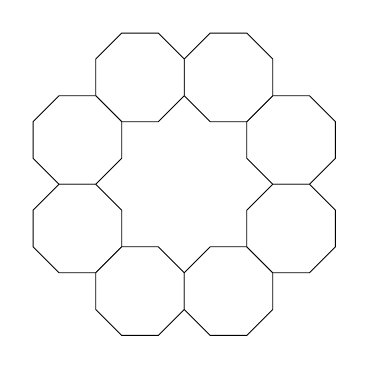}\quad
	\includegraphics[scale=0.4]{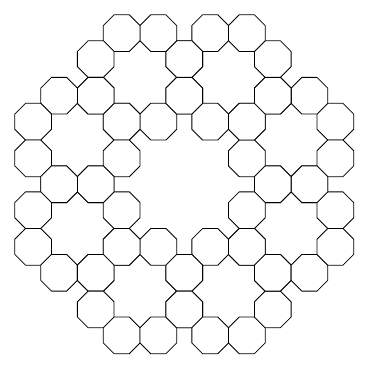}\quad
	\includegraphics[scale=0.4]{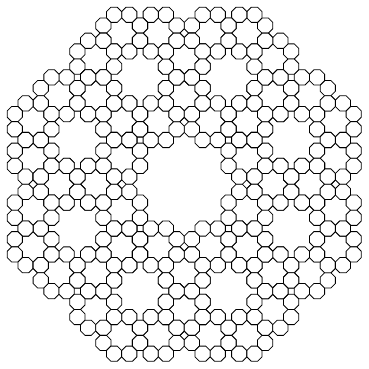}
	\caption{Octagons generated by the ``central'' octagon in $P_{1}$.}
	\label{fig:OctCentralPackings}
	\end{center}
	\end{figure}

Clearly, $P_{n}^{*} + \Gamma = P_{n}$, and 
$\displaystyle\bigcup_{k=1}^{n} P_{k}^{*} + \Gamma = 
\displaystyle\bigcup_{k=1}^{n} P_{k}$. Observe as well that 
$\sym(P_{n}^{*})=D_{8}$ for all $n$. We will construct for each $n$ a compact 
subset $F_{n}$ of $P_{n}^{*}$ satisfying the following properties:

	\begin{enumerate}[S1.]
	\item For every $n$, $\sym(F_{n})=D_{8}$.
	\item For every $n$, $F_{n}+\Gamma=\displaystyle\bigcup_{k=1}^{n}P_{k}$.
	\item For every $n$ and any nontrivial $\mathbf{v} \in \Gamma$, 
	$\intr(F_{n}) \cap \intr(\mathbf{v}+F_{n}) = \varnothing$.
	\item The sequence $\set{F_{n}}_{n=1}^{\infty}$ is Cauchy in $H(\R^{2})$,
	the space of non-empty compact subsets of $\R^{2}$ equipped with the induced 
	Hausdorff metric.
	\end{enumerate}
	
If there exist such $F_{n}$, let $F_{\square}=\dlim_{n \rightarrow \infty}  F_{n}$, 
which is well-defined by completeness of $H(\R^{2})$. By conditions S2 and S3, and 
the fact that $\cl\left(\displaystyle\bigcup_{k=1}^{\infty} P_{k}\right)=\R^{2}$, 
$F_{\square}$ is a fundamental domain for $\Gamma$. Condition S1 and continuity of 
isometries in $H(X)$ imply that $\sym\left(F_{\square}\right)=D_{8}$.

To this end, color $P_{1}^{*}$ red and let $R_{1}=P_{1}^{*}$ and $F_{1}=R_{1}$. 
Then, $F_{1}$ satisfies S1, S2, and S3. At any succeeding step, an octagon $O$ with 
center $x$ will be colored purely red, purely white, or in the following manner: 
First, divide $O$ into eight congruent slices, namely 
$S_{i} = O \cap \cone(x;e_{i},e_{i+1})$ for $0 \leq i \leq 7$. Color $S_{k}$, 
where $k$ is even, all red or all white, and color the rest of the slices oppositely. 
Thus, the slices are colored in alternating fashion. See Figure~\ref{fig:ColorsOct}.

	\begin{figure}[ht]
	\begin{center}
	\includegraphics[scale=0.6]{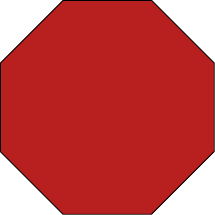}\qquad
	\includegraphics[scale=0.6]{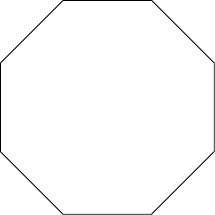}\qquad
	\includegraphics[scale=0.6]{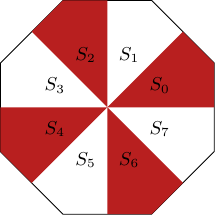}\qquad
	\includegraphics[scale=0.6]{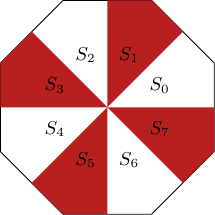}
	\caption{All possible colorings of any octagon at any step.}
	\label{fig:ColorsOct}
	\end{center}
	\end{figure}

After coloring the octagons in step $n$ according to the rule to be described 
below, call the union of the red pieces $R_{n}$, and define 
$F_{n} = \cl\left[(F_{n-1}\setminus P_{n}^{*})\cup R_{n}\right]$.

We now describe the coloring procedure for step $n$, where $n \geq 2$, 
with the assumption that all octagons in previous steps were colored in one of 
the four ways in Figure~\ref{fig:ColorsOct}. If $O$ is an octagon in $P_{n}$ 
with center $x$, then $x$ is the vertex of either one or two octagons in 
$P_{n-1}^{*}$, as in Figure~\ref{fig:OctAlg}.

	\begin{figure}[ht]
	\begin{center}
	\includegraphics[scale=0.8]{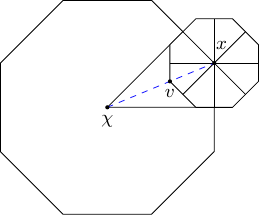} \qquad
	\includegraphics[scale=0.8]{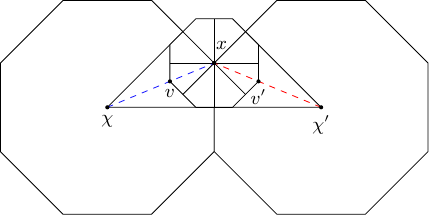}
	\caption{An octagon may be centered at an unshared or a shared vertex.}
	\label{fig:OctAlg}
	\end{center}
	\end{figure}

Suppose there is exactly one such octagon $\mathcal{O}$, with center $\chi$. 
Now, there is a unique vertex $v$ of $O$ such that $v$ lies on $(x,\chi)$. 
Let $S$ be the slice of $O$ containing $v$, that is, 
$v \in S \coloneqq O \cap \cone(x;e_{i},e_{i+1})$ for some unique 
$i \in \{0,1,2,\ldots,7\}$. Then, because $(x,\chi)$ is contained in a slice 
of $\mathcal{O}$, either $(x,\chi) \subseteq F_{n-1}$ or $(x,\chi)$ is disjoint 
with $F_{n-1}$. Furthermore, $\intr(O \setminus \mathcal{O})$ is contained in 
either $\R^{2} \setminus \displaystyle\bigcup_{k=1}^{n-1} P_{k}^{*}$, or in a 
portion of a slice of an octagon in some previous $P_{k}^{*}$, $k<n-1$, that has 
not been re-colored after the $k$th step. In either case, either 
$\intr(O \setminus \mathcal{O}) \subseteq F_{n-1}$ 
or $\intr(O \setminus \mathcal{O})$ is disjoint with $F_{n-1}$.

We divide $O$ into two subsets 
$\displaystyle\bigcup_{i=0}^{3} O \cap \cone(x;e_{2i},e_{2i+1})$ and 
$\displaystyle\bigcup_{i=0}^{3} O \cap \cone(x;e_{2i+1},e_{2i+2})$ with 
non-overlapping interiors. As mentioned previously, each of these will be 
colored purely red or purely white. If the subset containing $(x,v)$ is denoted 
by $\mathcal{C}_{v}$, then the other subset is 
$\cl\left(O \setminus \mathcal{C}_{v}\right)$. 

	\begin{enumerate}[1.]
	\item Color $\mathcal{C}_{v}$ red if and only if $(x,\chi) \subseteq F_{n-1}$.
	\item Color $\cl\left(O \setminus \mathcal{C}_{v}\right)$ red if and only if 
	$\intr(O \setminus \mathcal{O}) \subseteq F_{n-1}$.
	\end{enumerate}
	
In other words, if $(x,\chi)$ and $\intr(O \setminus \mathcal{O})$ have the same 
color after step $n-1$, then in step $n$ we color $O$ the same way. On the other 
hand, if $(x,\chi)$ and $\intr(O \setminus \mathcal{O})$ are oppositely colored 
after step $n-1$, we color the slices of $O$ alternately red and white such that 
$S$ inherits the color of $(x,\chi)$.

Suppose there are two octagons in $P_{n-1}^{*}$ having $x$ as a vertex, say 
$\mathcal{O}$ and $\mathcal{O}'$ with centers $\chi$ and $\chi '$, respectively. 
Let $v$ and $v'$ be the vertices of $O$ lying on $(x,\chi)$ and $(x,\chi ')$, 
respectively, and $\mathcal{C}_{v}$ and 
$\cl\left(O \setminus \mathcal{C}_{v}\right)$ be as before. 

	\begin{enumerate}[1.]
	\item Color $\mathcal{C}_{v}$ red if and only if $(x,\chi) \subseteq F_{n-1}$.
	\item Color $\cl\left(O \setminus \mathcal{C}_{v}\right)$ red if and only if 
	$(x,\chi') \subseteq F_{n-1}$.
	\end{enumerate}
	
This is well-defined because $v'$ is either in $O \cap \cone(x;e_{i+3},e_{i+4})$ or 
$O \cap \cone(x;e_{i-3},e_{i-2})$.

To demonstrate, in step 2, let $O$ be any octagon in $P_{2}^{*}$. Refer to 
Figure~\ref{fig:Step2OctCentral}. Then, its center is an unshared vertex, and 
$\mathcal{O}=F_{1}=P_{1}^{*}$, and 
$\intr(O \setminus \mathcal{O}) \subseteq \R^{2} \setminus P_{1}^{*}$. Thus, the 
associated $(x,\chi) \subseteq F_{1}$ and $\intr(O \setminus \mathcal{O})$ is 
disjoint with $F_{1}$. Therefore, we color $\mathcal{C}_{v}$ red and 
$\cl\left(O \setminus \mathcal{C}_{v}\right)$ white.  Note that $F_{2}$ satisfies 
S1, S2, and S3. 

	\begin{figure}[ht]
	\begin{center}
	\includegraphics[scale=0.6]{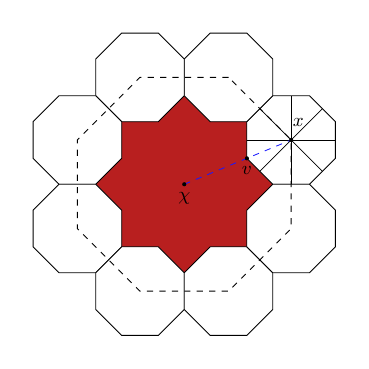} \qquad
	\includegraphics[scale=0.6]{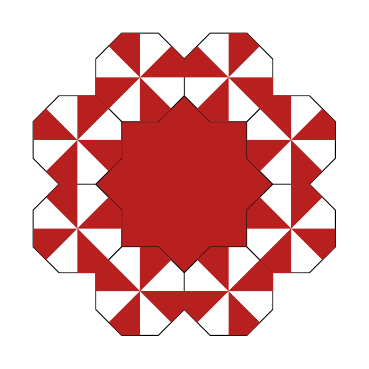}
	\caption{Second step of construction of a highly symmetric fundamental domain 
	for $\mathbb{Z}^{2}$.}
	\label{fig:Step2OctCentral}
	\end{center}
	\end{figure}

For $n \geq 3$, assume that $F_{n-1}$ satisfies S1, S2, and S3. We prove that 
so does $F_{n}$. First, we show that it satisfies S1. Because $\sym(F_{n-1})=D_{8}$, 
then for any $\sigma \in D_{8}$ and octagon $O$ in $P_{n}^{*}$, $(x,\chi)$ and 
$\sigma(x,\chi)$ have the same color in the preceding step. The same is true of 
$\intr(O \setminus \mathcal{O})$ and $\sigma(\intr(O \setminus \mathcal{O}))$ 
for unshared octagons, and of $(x,\chi')$ and $\sigma(x,\chi')$ for shared 
octagons. Thus, for any pair $O$ and $\sigma(O)$, the octagons will be colored 
such that the red pieces are invariant under $\sigma$. From this, 
$\sym(R_{n})=D_{8}$ and it follows that $\sym(F_{n})=D_{8}$.

We now prove that $F_{n}$ satisfies S2 and S3. Let $O$ in $P_{n}^{*}$ with center 
$x$ and consider all $\mathbf{v}+O$, $\mathbf{v} \in \Gamma$, such that 
$\mathbf{v}+O$ is also in $P_{n}^{*}$. Then, exactly one of the following 
is true:

	\begin{enumerate}[(1.)]
	\item There exists a unique octagon $\mathcal{O}$ in $P_{n-1}^{*}$ such 
	that for all such $\mathbf{v}$, $\mathbf{v}+\mathcal{O}$ is the unique octagon 
	in $P_{n-1}^{*}$ having $\mathbf{v}+x$ as a vertex. For example, in 
	Figure~\ref{fig:TranslatesOct}, for each $\mathbf{v}$ such that $\mathbf{v}+A$ 
	is in $P_{3}^{*}$, the center of $\mathbf{v}+A$ is a vertex of 
	$\mathbf{v}+\mathcal{O}$ and of no other octagon in $P_{2}^{*}$. 
	\item There exist exactly two octagons $\mathcal{O}$ and $\mathcal{O}'$ in 
	$P_{n-1}^{*}$ such that for all such $\mathbf{v}$, $\mathbf{v}+x$ is a vertex 
	shared by $\mathbf{v}+\mathcal{O}$ and $\mathbf{v}+\mathcal{O}^{*}$ in 
	$P_{n-1}^{*}$. In Figure~\ref{fig:TranslatesOct}, for each $\mathbf{v}$ such 
	that $\mathbf{v}+B$ is in $P_{3}^{*}$, the center of $\mathbf{v}+B$ is a 
	shared vertex of $\mathbf{v}+\mathcal{O}$ and $\mathbf{v}+\mathcal{O}_{B}$.
	\item There exist octagons $\mathcal{O}$ and $\mathcal{O}'$ such that for all 
	such $\mathbf{v}$, $\mathbf{v}+x$ is a vertex of $\mathbf{v}+\mathcal{O}$ or 
	$\mathbf{v}+\mathcal{O}^{*}$ in $P_{n-1}^{*}$, but there exists $\mathbf{v}^{*}$ 
	such that exactly one of $\mathbf{v}^{*}+\mathcal{O}$ and 
	$\mathbf{v}^{*}+\mathcal{O}'$ is in $P_{n-1}^{*}$. In 
	Figure~\ref{fig:TranslatesOct}, note that the center of $C$ is shared by 
	$\mathcal{O}$ and $\mathcal{O}_{C}$ but the center of $\vect{-1}{0}+C$ is an 
	unshared vertex of $\vect{-1}{0}+\mathcal{O}$. 
	\end{enumerate}

	\begin{figure}[ht]
	\begin{center}
	\includegraphics[scale=0.6]{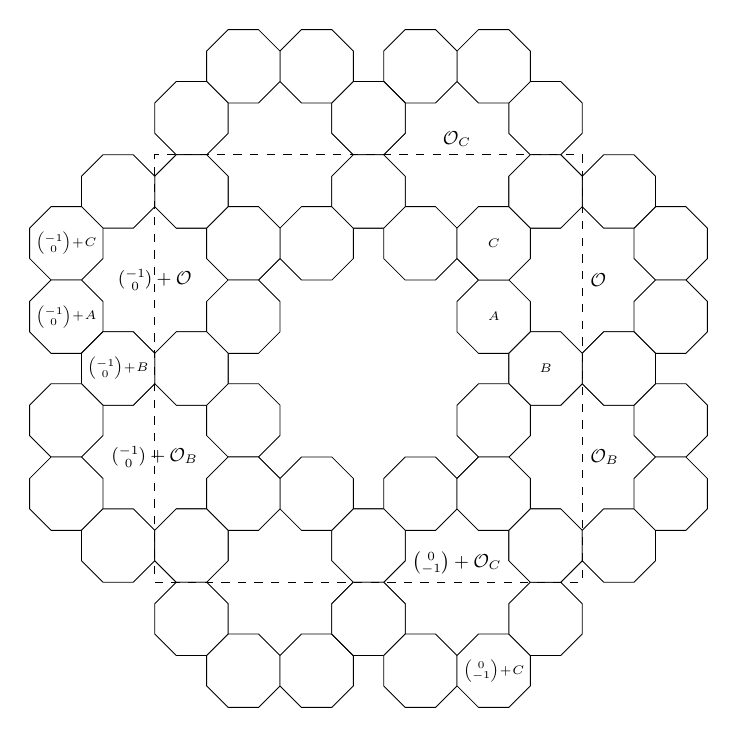} 
	\caption{Selected octagons with their translates in $P_{3}^{*}$.}
	\label{fig:TranslatesOct}
	\end{center}
	\end{figure}

We consider each case.

	\begin{enumerate}[(1.)]
	\item By S2 and S3, each of the associated $(x,\chi)$ and 
	$\intr(O \setminus \mathcal{O})$ has exactly one translate that is red in 
	the previous step. Thus, in the current step, each of $\mathcal{C}_{v}$ and 
	$\cl\left(O \setminus \mathcal{C}_{v}\right)$ will have exactly one translate 
	that will be colored red.
	\item Similarly, each of the associated $(x,\chi)$ and $(x,\chi')$ has exactly 
	one translate that is red in the previous step. 
	\item Consider $\mathbf{v}_{1}$ and $\mathbf{v}_{2}$ such that 
	$\mathbf{v}_{1}+\mathcal{O}$ and $\mathbf{v}_{2}+\mathcal{O}'$ are in 
	$P_{n-1}^{*}$, as in Figure~\ref{fig:OctCase3A}. Here, 
	$\mathbf{v}_{1}+\mathcal{O}'$ and $\mathbf{v}_{2}+\mathcal{O}$ are not 
	necessarily in $P_{n-1}^{*}$. Let $\chi$ and $\chi '$ be the centers of 
	$\mathcal{O}$ and $\mathcal{O}'$, respectively. Let $v$ and $v'$ be the vertices 
	of $O$ such that $\mathbf{v}_{1}+v$ lies on $\mathbf{v}_{1}+(x,\chi)$ and 
	$\mathbf{v}_{2}+v'$ lies on $\mathbf{v}_{2}+(x,\chi ')$. We note that 
	$\mathbf{v}_{1}+(x,v') \subseteq \mathbf{v_{1}}+\intr(O \setminus \mathcal{O})$. 
	Thus, either there is exactly one translate of $\mathcal{O}'$ such that the 
	slice containing the corresponding translate of $(x,v')$ is red, or there is 
	exactly one translate of $\intr(O \setminus \mathcal{O})$ that is contained 
	in a red slice of some octagon in an earlier step. This implies that 
	$\cl\left(O \setminus \mathcal{C}_{v}\right)$ will have exactly one translate 
	that will be colored red. Similarly, exactly one translate of 
	$\mathcal{C}_{v}$ will be colored red.

	\begin{figure}[ht]
	\begin{center}
	\includegraphics[scale=0.8]{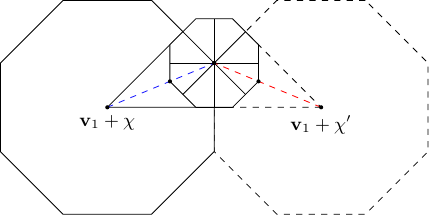} \qquad
	\includegraphics[scale=0.8]{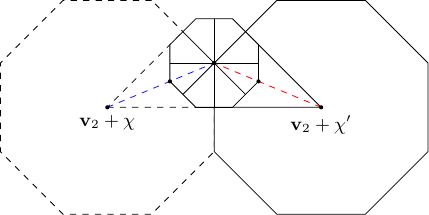}
	\caption{The case when a translate of an octagon in $P_{n-1}^{*}$ is not 
	necessarily in $P_{n-1}^{*}$.}
	\label{fig:OctCase3A}
	\end{center}
	\end{figure}

	\end{enumerate}

This shows that for any octagon $O$ in $P_{n}^{*}$ with center $x$, exactly one 
translate of each of  $\displaystyle\bigcup_{i=0}^{3} O \cap \cone(x;e_{2i},e_{2i+1})$ 
and $\displaystyle\bigcup_{i=0}^{3} O \cap \cone(x;e_{2i+1},e_{2i+2})$ will be 
colored red. Then, $R_{n} + \Gamma = P_{n}$ and 
$\intr(R_{n}) \cap \intr(\mathbf{v}+R_{n}) = \varnothing$ 
for any nontrivial $\mathbf{v}$. It follows that $F_{n}$ satisfies S2 and S3. 
In Figure~\ref{fig:Step34OctCentral}, we illustrate the third and fourth 
construction steps. 

	\begin{figure}[ht]
	\begin{center}
	\includegraphics[scale=0.55]{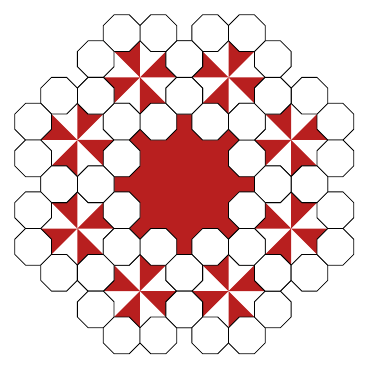} \quad
	\includegraphics[scale=0.55]{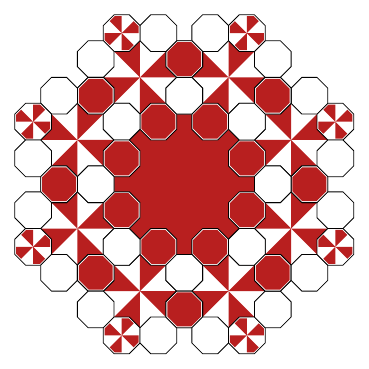}\quad
	\includegraphics[scale=0.55]{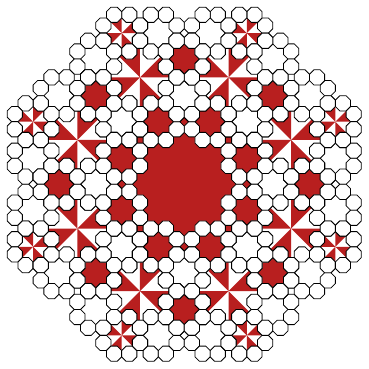} \quad
	\includegraphics[scale=0.55]{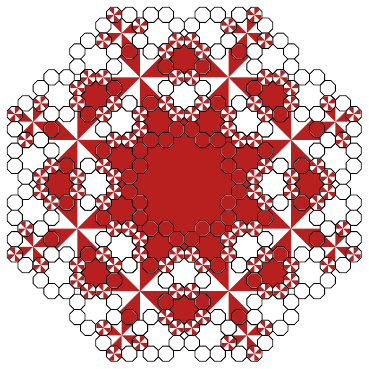}
	\caption{Third and fourth steps of construction of a highly symmetric 
	fundamental domain for $\mathbb{Z}^{2}$.}
	\label{fig:Step34OctCentral}
	\end{center}
	\end{figure}

We note that $F_{n}$ is compact, as it is closed by definition, and 
$\diam(F_{n})=  \dsum_{k = 1}^{n} h\ell^{k-1}$, where 
$h=2\sqrt{1-\frac{\sqrt{2}}{2}}$. Furthermore, in the Hausdorff metric-equipped 
space $H(\mathbb{R}^{2})$ of non-empty compact subsets of $\mathbb{R}^{2}$, the 
sequence $\set{F_{n}}_{n=1}^{\infty}$ is Cauchy, seeing that the distance between 
$F_{n}$ and $F_{m}$ for $n>m$ is $\dsum_{k = m+1}^{n} h\ell^{k-1}$. Here, it is 
also important to note that for an octagon $O$ with center $x$ in $P_{n}^{*}$, 
by construction, $\displaystyle\bigcup_{i=n+1}^{\infty}P_{i}^{*}$ will not cover 
$O$, with $O \setminus \displaystyle\bigcup_{i=n+1}^{\infty}P_{i}^{*}$ being a 
connected set containing the ball centered at $x$ with radius 
$1-\frac{\sqrt{2}}{{2}}$ times that of $O$. This means that 
$O \setminus \displaystyle\bigcup_{i=n+1}^{\infty}P_{i}^{*}$ will not be 
affected by any succeeding step. Because this is true for any octagon at any 
stage, we find that in the limit, the union of the portion of each  
$O \setminus \displaystyle\bigcup_{i=n+1}^{\infty}P_{i}^{*}$ that overlaps with 
the unit square fundamental domain covers this unit square up to a set of 
measure zero. From this, the boundary of the limit $F_{\square}$ of 
$\set{F_{n}}_{n=1}^{\infty}$ in $H(\mathbb{R}^{2})$ has measure zero.

We conclude that $F_{\square}$ is a fundamental domain for $\Gamma$ with the 
desired symmetry group $D_{8}$. See Figure~\ref{fig:D8forSquare} for the tile 
$F$ together with some of its translates.

\begin{figure}[ht]
\epsfig{file=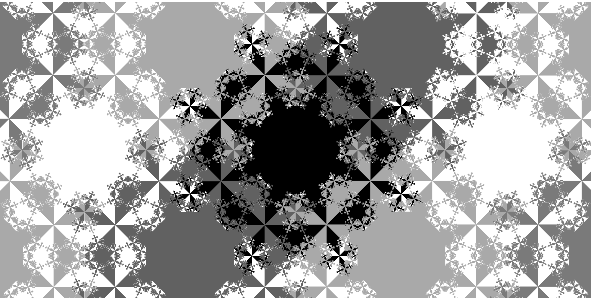, width=130mm}
\caption{The fundamental domain $F_{\square}$ (black) of $\Z^2$ and some
  of its copies, illustrating how they form a tiling \label{fig:D8forSquare}.} 
\end{figure}

\end{proof}

\begin{prop}[Elser] \label{zwei}
The hexagonal lattice $A_2$ has a compact fundamental domain
$F_{\triangle}$ with $S(F_{\triangle}) = D_{12}$. 
\end{prop}

\begin{proof}[Proof (sketch)]

The hexagonal lattice case is analogously treated. The plane is first packed 
by dodecagons inscribed in hexagonal fundamental domains for the hexagonal 
lattice $A_{2}$ defined by $\begin{pmatrix} 1 \\ 0 \end{pmatrix}$ and 
$\begin{pmatrix} \cos \frac{\pi }{3} \\ \sin \frac{\pi }{3} \end{pmatrix}$ 
(corresponding to the Archimedean tiling $3.12^2$ by triangles and 
dodecagons). Further generations are obtained by placing dodecagons centered 
at the vertices of the dodecagons of the preceding generation. Here, the 
``holes'' of the union of the first $n$ packings are always equilateral 
triangles, and these also vanish eventually in the progression.  A central 
dodecagon and the dodecagons arising from it are considered, and subsets are 
taken at each step analogously as in the coloring procedure described for 
the square lattice case.  See Figure~\ref{fig:D12forHexagon}.

	\begin{figure}
	\parbox[c]{81mm}{\epsfig{file=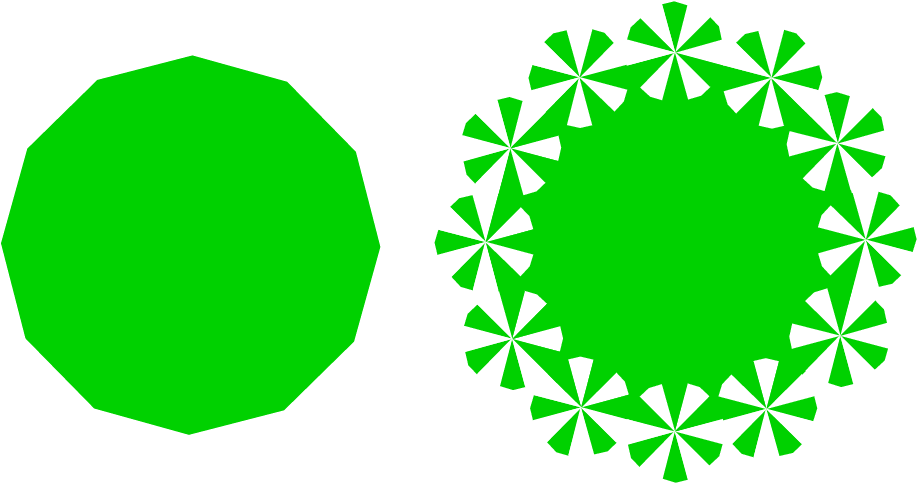,width=80mm}}
	\parbox[c]{51mm}{\epsfig{file=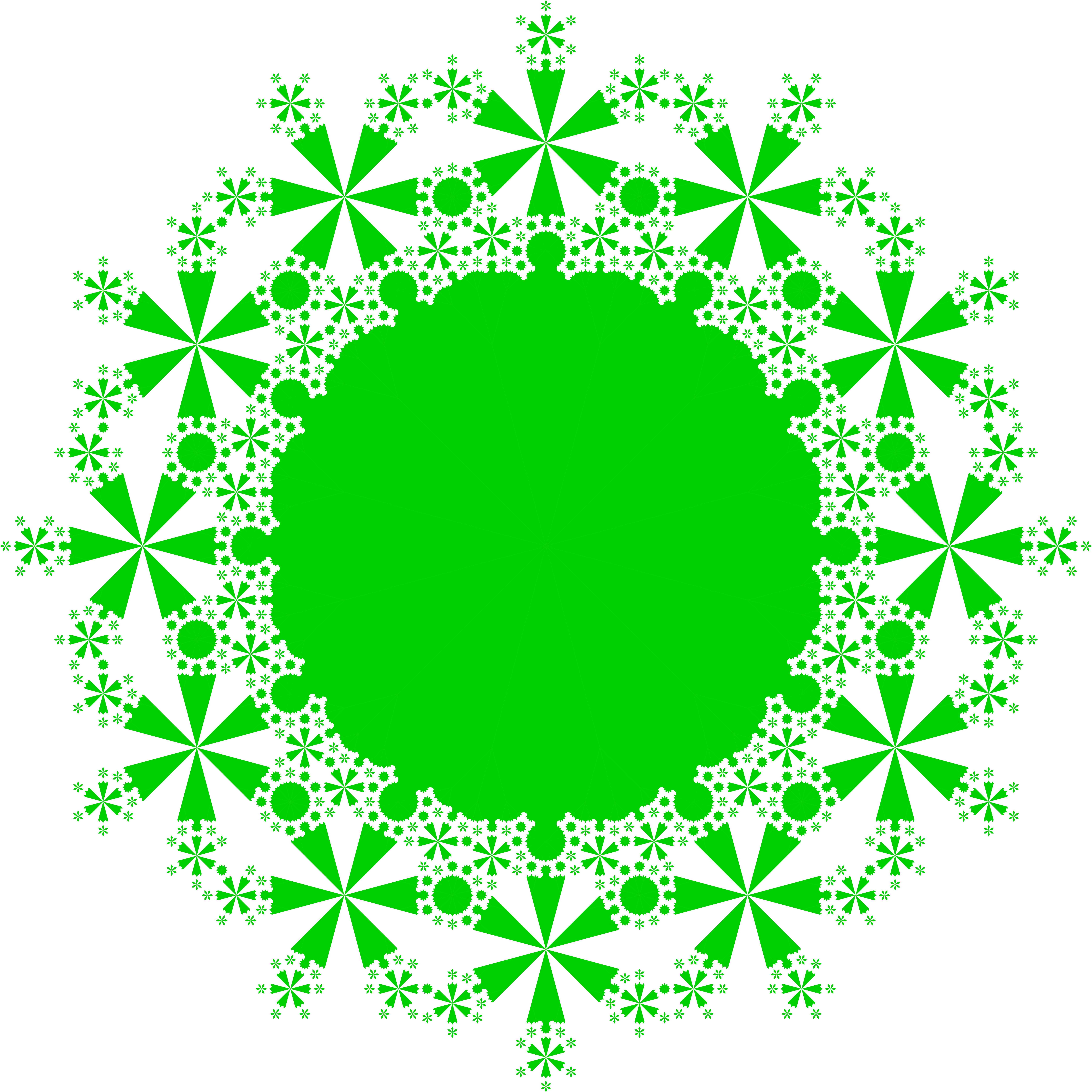, width=50mm}}
	\caption{The first two iterates of the construction of the 12-fold
	  fundamental domain $F_{\triangle}$ of the hexagonal lattice $A_2$
	  (left and middle), and a higher iterate (right). \label{fig:D12forHexagon}}
	\end{figure}

Interestingly, the set $F_{\triangle}$ appears in an entirely
different context in \cite{bks} and \cite{coc}, where it serves
as a {\em window} respectively {\em atomic surface} for 
mathematical quasicrystals. The tiling property of $F_{\triangle}$
is not mentioned in these texts, and not obvious from the constructions
used there.

\end{proof}

Let us now prove Theorem \ref{thm:satz1}.

\begin{proof}[Proof (of Theorem \ref{thm:satz1})]
We consider the four cases where $\Gamma$ is an oblique lattice, 
a square lattice, a hexagonal lattice, or a rectangular lattice.

{\bf Case 1:} Let $\Gamma$ be an oblique lattice.
Without loss of generality one basis of $\Gamma$ is $b_1 =  \big(
\begin{smallmatrix}   x \\ 0 \end{smallmatrix} \big), \; b_2=
\big( \begin{smallmatrix} y   \\ z \end{smallmatrix} \big)$, 
$z \ne 0$. 
Then, let $F$ be the rectangle with vertices 
$\big( \begin{smallmatrix} x/2 \\ z/2 \end{smallmatrix} \big), \,  
\big( \begin{smallmatrix} -x/2 \\ z/2 \end{smallmatrix} \big), \, 
\big( \begin{smallmatrix} -x/2 \\ -z/2 \end{smallmatrix} \big), \, 
\big( \begin{smallmatrix} x/2 \\ -z/2 \end{smallmatrix} \big)$, 
see Figure \ref{fig:ob}.  
 
\begin{figure}[ht!]
\begin{center}
\includegraphics[scale=1]{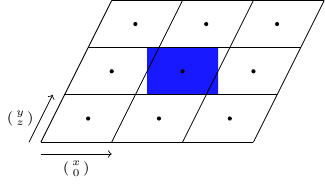}
\caption{A fundamental domain for an oblique lattice, with $D_2$-symmetry.}
\label{fig:ob}
\end{center}
\end{figure} 
 
It is easy to see that $\Gamma+F = \R^2$, and the copies of
$F$ do not overlap. Thus $F$ is a fundamental domain of $\Gamma$. 
We have $P(\Gamma)=C_2$ and $S(F)=D_2$, unless $x=z$, in which case $S(F)=D_4$. 
In the last case, one can use a rectangle with one edge having the same length
as and parallel to $\big( \begin{smallmatrix} y \\ z \end{smallmatrix} \big)$ 
and suitable perpendicular edge length to obtain $S(F)=D_{2}$.

{\bf Case 2:} $\Gamma = \Z^2$: This is Proposition \ref{eins}. 
We have $P(\Gamma)=D_4$ and $S(F_{\square})=D_8$. 

{\bf Case 3:} $\Gamma = A_2$: This is Proposition \ref{zwei}. 
We have $P(\Gamma)=D_6$ and $S(F_{\triangle})=D_{12}$.

{\bf Case 4:} The algorithm to be presented for the rectangular lattice case is 
for the most part analogous to those of the square and hexagonal cases. The idea 
is to cover more and more of the plane such that squares of a current generation 
are centered at the vertices of those of the previous generation. The sequences 
of sizes of the squares remain nondecreasing,  but this time the squares are 
allowed to be of the same size as those in the previous generation. Thus we 
distinguish between ``iterations'' and ``steps''. One step consists of one or 
more iterations where squares of the same size are used, and squares in a current 
step are smaller than those in the previous step.

Let $\Gamma$ be a rectangular lattice with basis vectors 
$\vect{b}{0}$, $\vect{0}{a}$, with $b<a$. Let $r_{0}=a$, $r_{1}=b$, 
$v_{1}=\left\lfloor\frac{r_{0}}{r_{1}}\right\rfloor$, $r_{2}=r_{0}-v_{1}r_{1}$. 
The first step shall consist of $v_{1}$ iterations, and squares of 
edge length $r_{1}$ will be used. Let $P_{1,1}$ be the packing by squares of edge 
length $b$ centered at each lattice point. If $v_{1}=1$, then the next packing 
belongs to the second step. Otherwise, for each $k$ from $2$ to $v_{1}$, let 
$P_{1,k}$ be the collection of squares of edge length $r_{1}$ centered at the 
vertices of the squares of $P_{1,k-1}$.

If $r_{2}=0$, then there are no further steps to consider. Otherwise, the portion 
of the plane that remains uncovered by the squares may be viewed as a union of 
$r_{1}/2 \times r_{2}/2$ rectangles.  If $r_{1} \times r_{1}$ squares are placed 
at the vertices of the current packing, then there would be nontrivial overlapping 
regions. Thus, smaller squares will be used. 

In general, suppose $r_{j} \neq 0$, where $j \geq 2$, and the uncovered portions 
are $r_{j-1}/2 \times r_{j}/2$ rectangles. Let 
$v_{j}=\left\lfloor\frac{r_{j-1}}{r_{j}}\right\rfloor$, 
$r_{j+1}=r_{j-1}-v_{j}r_{j}$.  For each $k$ from $1$ to $v_{j}$, let $P_{j,k}$ 
be the collection of squares of edge length $r_{j}$ centered at the vertices of 
the squares in $P_{j,k-1}$, where we define $P_{j,0}$ naturally as 
$P_{j-1,v_{j-1}}$. After the $j$th step, the uncovered portion of the plane, 
if any, consists of $r_{j}/2 \times r_{j+1}/2$ rectangles. Note that the sequence 
of $r_{j}$'s is the output of the Euclidean algorithm. Thus, $r_{n+1}=0$ for some 
$n$ if and only if $\frac{a}{b} \in \mathbbm{Q}$, and the process terminates after 
step $n$ in this case. In the case that $\frac{a}{b} \in \mathbbm{Q}'$, the 
sequence of $r_{j}$'s converges to $0$, so that the process also produces squares 
that cover the plane in the limit. 

As before, let $P_{1,1}^{*}$ be the square of edge length $r_{1}$ centered at the 
origin, and for each $j$, $k$, let $P_{j,k}^{*}$ consist of the squares of $P_{j,k}$ 
centered at the vertices of $P_{j,k-1}^{*}$. Again, $P_{j,k}^{*} + \Gamma = P_{j,k}$, 
$\displaystyle\bigcup_{j} \bigcup_{k=1}^{v_{j}} P_{j,k}^{*} + \Gamma = 
\displaystyle\bigcup_{j}\bigcup_{k=1}^{v_{j}} P_{j,k}$, and 
$\sym(P_{j,k}^{*})=D_{4}$ for all $j$, $k$. We will construct for each $j$, $k$ a 
subset $F_{j,k}$ of $P_{j,k}^{*}$ satisfying the following properties:

	\begin{enumerate}
	\item[R1.] For every $j$, $k$, $\sym(F_{j,k})=D_{4}$.
	\item[R2.] For every $j$, $k$, 
	$F_{j,k}+\Gamma=\displaystyle\bigcup_{j}\bigcup_{k=1}^{v_{j}}P_{j,k}$.
	\item[R3.] For every $j$, $k$ and any nontrivial $\mathbf{v} \in \Gamma$, 
	$\intr(F_{j,k}) \cap \intr(\mathbf{v}+F_{j,k}) = \varnothing$.
	\item[R4.] The sequence $\set{F_{j,k}}$ is Cauchy in $H(\R^{2})$.
	\end{enumerate}
	
Color $P_{1,1}^{*}$ red and let $R_{1,1}=P_{1,1}^{*}$ and $F_{1,1}=R_{1,1}$. 
This region satisfies R1, R2, and R3. Similarly, any square $O$ with center $x$ 
will be colored purely red, colored purely white, or divided into four congruent 
slices $S_{i} = O \cap \cone(x;e_{i},e_{i+1})$ for $i=0,1,2,3$ and have two 
non-adjacent slices colored red. Here, 
$e_{i} = \begin{pmatrix} \cos \frac{\pi i}{2} \\ \sin \frac{\pi i}{2} \end{pmatrix}$. 
The union of the red pieces at the $k$th iteration of the $j$th step will be 
denoted by $R_{j,k}$, and $F_{j,k}$ is taken to be 
$\cl\left[(F_{j,k-1}\setminus P_{j,k}^{*})\cup R_{j,k}\right]$.

This time, if a square $O$ in $P_{j,k}^{*}$ has center $x$, then $x$ is the 
vertex of either one, two, or four squares in $P_{j,k-1}^{*}$. For the first two 
cases, we apply the natural analogues of the procedures for the square and 
hexagonal lattices. In the last case, for every vertex $v$ of $O$ and every 
square $\mathcal{O}$ having $x$ as vertex, color the slice of $O$ containing 
$(x,v)$ according to the color of $(x,\chi)$ in the previous iteration, where 
$\chi$ is the center of $\mathcal{O}$. In simpler 
terms, the slices of $O$ retain the way they are colored previously.

The fact that $\sym(F_{j,k})=D_{4}$ for all $j$, $k$ is proved similarly as 
in the previous cases. We now prove that $F_{j,k}$ satisfies R2 and R3. 
Let $O$ be in $P_{j,k}^{*}$ with center $x$ and consider all  
$\mathbf{v} \in \Gamma$ such that $\mathbf{v}+O$ is also in $P_{j,k}^{*}$. 
Then, exactly one of the following is true:

	\begin{enumerate}[(1.)]
	\item There exist either one, two, or four squares in $P_{j,k-1}$ such 
	that for all such $\mathbf{v}$, $\mathbf{v}+x$ is a vertex of the 
	translates by $\mathbf{v}$ of the one, two, or four squares, respectively, 
	and only of these squares.
	\item There exist squares $\mathcal{O}$ and $\mathcal{O}'$ such that for 
	all such $\mathbf{v}$,  $\mathbf{v}+x$ is a vertex of $\mathbf{v}+\mathcal{O}$ 
	or $\mathbf{v}+\mathcal{O}^{'}$ in $P_{j,k-1}^{*}$, but there exists 
	$\mathbf{v}^{*}$ such that exactly one of $\mathbf{v}^{*}+\mathcal{O}$ and 
	$\mathbf{v}^{*}+\mathcal{O}'$ is in $P_{j,k-1}^{*}$. 
	\item There exist four squares $\mathcal{O}_{1}$, $\mathcal{O}_{2}$, 
	$\mathcal{O}_{3}$, $\mathcal{O}_{4}$ such that for all such $\mathbf{v}$, 
	$\mathbf{v}+x$ is a vertex in $P_{j,k-1}^{*}$ of either the pair 
	$\mathbf{v}+\mathcal{O}_{1}$ and $\mathbf{v}+\mathcal{O}_{2}$, or 
	$\mathbf{v}+\mathcal{O}_{3}$ and $\mathbf{v}+\mathcal{O}_{4}$, or both pairs, 
	but there exists $\mathbf{v}^{*}$ such that exactly one pair has a translate 
	by $\mathbf{v}^{*}$ in $P_{j,k-1}^{*}$.
	\end{enumerate}

The first two cases are dealt with analogously as in the square and hexagonal cases. 
As for the third case, if $\mathbf{v}+x$ is a vertex of 
$\mathbf{v}+\mathcal{O}_{3}$ and $\mathbf{v}+\mathcal{O}_{4}$ but not of 
$\mathbf{v}+\mathcal{O}_{1}$ and $\mathbf{v}+\mathcal{O}_{2}$, then 
$\mathbf{v}+\mathcal{O}_{1}$ and $\mathbf{v}+\mathcal{O}_{2}$ must be in an 
``extreme'' portion of $P_{j,k-1}^{*}$, that is, in an extreme vertical or extreme 
horizontal position. Moreover, in $P_{j,k-1}^{*}$, either $k-1 \geq 2$ or $k-1=0$ 
and $P_{j,k-1}^{*}=P_{j-1,v_{j-1}}^{*}$ where $v_{j-1} \geq 2$. That is, 
$P_{j,k}^{*}$ comes after an iteration that is at least the second in its step. 
And because $\mathbf{v}+x$ is a shared vertex, it is not a ``corner'' of 
$P_{j,k-1}^{*}$. Thus, $\mathbf{v}+O$ will be colored purely white because from 
the coloring procedure, among the squares in the extreme of $P_{j,k-1}^{*}$, 
only the squares centered at the corners of $P_{j,k-1}^{*}$ will have red portions. 
This means that the non-corner squares in the extreme may be ignored, and the case 
is reduced to the first case where each translate of $O$ is shared by the 
corresponding translates of four squares. See for example 
Figure~\ref{fig:Shared2or4}.

	\begin{figure}[ht]
	\begin{center}
	\includegraphics[scale=0.85]{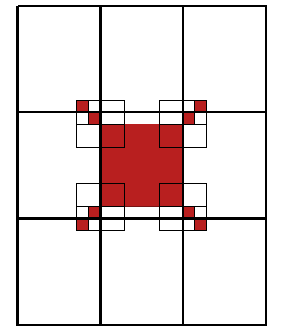}\quad
	\includegraphics[scale=0.85]{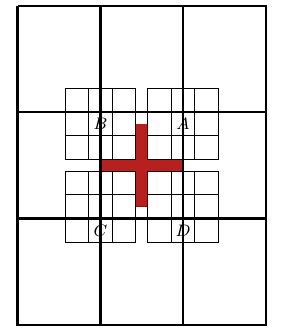}\quad
	\includegraphics[scale=0.85]{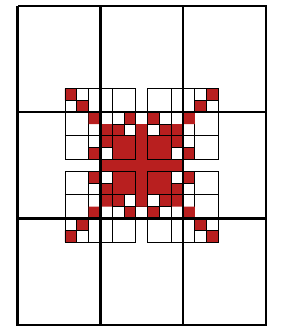}
	\caption{Squares $A$, $B$, $C$, and $D$ are equivalent under the action of 
	$\Gamma$, with $A$ and $B$ shared by four squares and $C$ and $D$ shared by 
	two squares.}
	\label{fig:Shared2or4}
	\end{center}
	\end{figure}

We note that the limit $F$ is compact even if $\frac{a}{b}$ is irrational, 
as the horizontal length of $F$ is $\displaystyle\sum_{j=1}^{\infty} v_{j}r_{j}  = 
\displaystyle\sum_{j=1}^{\infty} r_{j-1} - r_{j+1} =  
\displaystyle\lim_{j \rightarrow \infty} (r_{0}+r_{1}-r_{j}-r_{j+1}) = (a+b)$. 
If $\frac{a}{b}$ is rational, there exists a minimal $m$ such that $r_{m+1}=0$, 
and the horizontal length of $F$ is 
$\displaystyle\sum_{j=1}^{m} v_{j}r_{j} = a+b-r_{m}$. 
We thus conclude as in the square and hexagonal cases. Shown in 
Figure~\ref{fig:RecCons} are the steps in constructing a fundamental domain with 
the desired symmetry group for a particular rectangular lattice. 
Figure~$\ref{fig:RecTil}$ illustrates the tiling induced 
by the fundamental domain.

	\begin{figure}[ht]
	\begin{center}
	\includegraphics[scale=0.8]{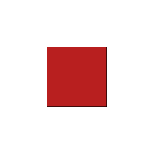}\quad
	\includegraphics[scale=0.8]{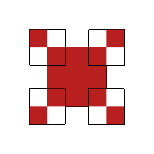}\quad
	\includegraphics[scale=0.8]{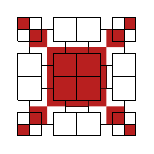}\quad
	\includegraphics[scale=0.8]{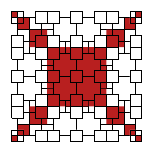}\quad
	\includegraphics[scale=0.8]{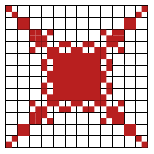}
	\caption{Construction of a $D_{4}$-symmetric fundamental domain for a 
	rectangular lattice with $a=8$, $b=5$.}
	\label{fig:RecCons}
	\end{center}
	\end{figure}

\begin{figure}[ht]
	\begin{center}
	\includegraphics[scale=1]{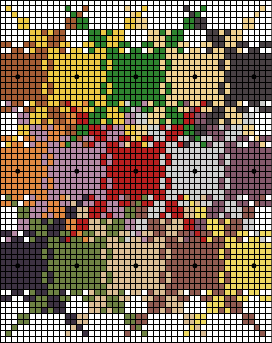}
	\caption{Tiling by a $D_{4}$-symmetric fundamental domain for a 
	rectangular lattice.}
	\label{fig:RecTil}
	\end{center}
	\end{figure}
\end{proof}

It is not obvious how to devise some similar general construction
of a fundamental domain for any rhombic lattice. 

\section{\bf Dimension 3} \label{sec:d3}

Similar to the proof of Theorem \ref{thm:satz1}, the proof of 
Theorem~\ref{thm:satz2} consists of considering all possible cases. Fortunately, 
we can utilize Theorem~\ref{thm:satz1} to cover most cases:  Create a right 
prism having one of the two-dimensional fundamental domains as base and with 
an appropriately chosen height. Then stack copies of this prism such that 
there is one such solid centered at each three-dimensional lattice point.  

In $\R^3$ there are 32 finite groups of Euclidean motions obeying the 
crystallographic restriction in Proposition~\ref{prop:crystrest}, 
see \cite{cox}, Section 15.6. Only seven of
them occur as point groups of lattices. Table \ref{tab1} summarises
the situation. The second column contains the name of the lattice, 
more precisely: the name of the family of lattices with a common symmetry 
group (the names as being used in crystallography). The third column 
contains the point group of the lattice in orbifold notation, the 
fourth column contains the order of the point group. 
The last column indicates the two-dimensional fundamental domain of 
Theorem~\ref{thm:satz1} which yields a three-dimensional fundamental
domain $F$ for the current three-dimensional lattice, and the order
$|S(F)|$ in parentheses.

\begin{table}[b]
\begin{tabular}{|l|l|c|c|c|}
\hline
Nr & Name & Point group & Order & 2-dim fundamental domain\\
& &     & & (number of symmetries $|S(F)|$) \\
\hline \hline
1 & $\Z^3$ & $\ast432$  & 48 & ---\\
2 & body centered cubic & $\ast432$ & 48 & ---\\
3 & face centered cubic & $\ast432$ & 48 & ---\\
\hline
4 & Hexagonal & $\ast622$ & 24 & 12fold (48)\\
\hline
5 & Tetragonal primitive & $\ast422$ & 16 & 8fold (32)\\
6 & Tetragonal body-centered & $\ast422$ & 16 & 8fold (32)\\
\hline
7 & Rhombohedral & $2\ast3$ & 12 & 6fold (24) / 12fold(48)\\ 
\hline
8 & Orthorhombic primitive & $\ast222$ & 8 & 4fold (16)\\
9 & Orthorhombic base-centered & $\ast222$ & 8 & 4fold (16)\\
10 & Orthorhombic body-centered & $\ast222$ & 8 & 4fold (16)\\
11 & Orthorhombic face-centered & $\ast222$ & 8 & 4fold (16)\\
\hline
12 & Monoclinic primitive & $2\ast$ & 4 & 2fold (8)/4fold(16)\\
13 & Monoclinic base-centered & $2\ast$ & 4 & 2fold (8)/4fold(16)\\
\hline
14 & Triclinic primitive & $2$ & 2 & mon.(4) / 2fold (8)\\
\hline
\end{tabular}
\caption{The 14 Bravais types of lattices and  their point groups. \label{tab1}}
\end{table}

Since the list of finite groups of Euclidean motions in $\R^3$
is known, we know that there is no such group containing 
the group $\ast432$ as a subgroup of index 2. (The only 
candidates---the ones of order 96---are the (non-primitive) groups 
$C_{96}, D_{48}$ and $C_2 \times D_{24}$, regarded as symmetry
groups of solids in $\R^3$.) 
The corresponding lattices are the so-called {\em cubic lattices}:
the {\em primitive cubic lattice} $\Z^3$, the {\em body centered cubic
lattice} $\Z^3 \cup \big(\Z^3 + (\frac{1}{2},\frac{1}{2},
\frac{1}{2})^T\big)$ (bcc) and the {\em face centered cubic lattice} 
(fcc). So we cannot expect to find fundamental domains for these three 
cubic lattices possessing a symmetry group that contains their point 
group $\ast 432$ as a proper subgroup of finite index.


\begin{figure}
\includegraphics[width=130mm]{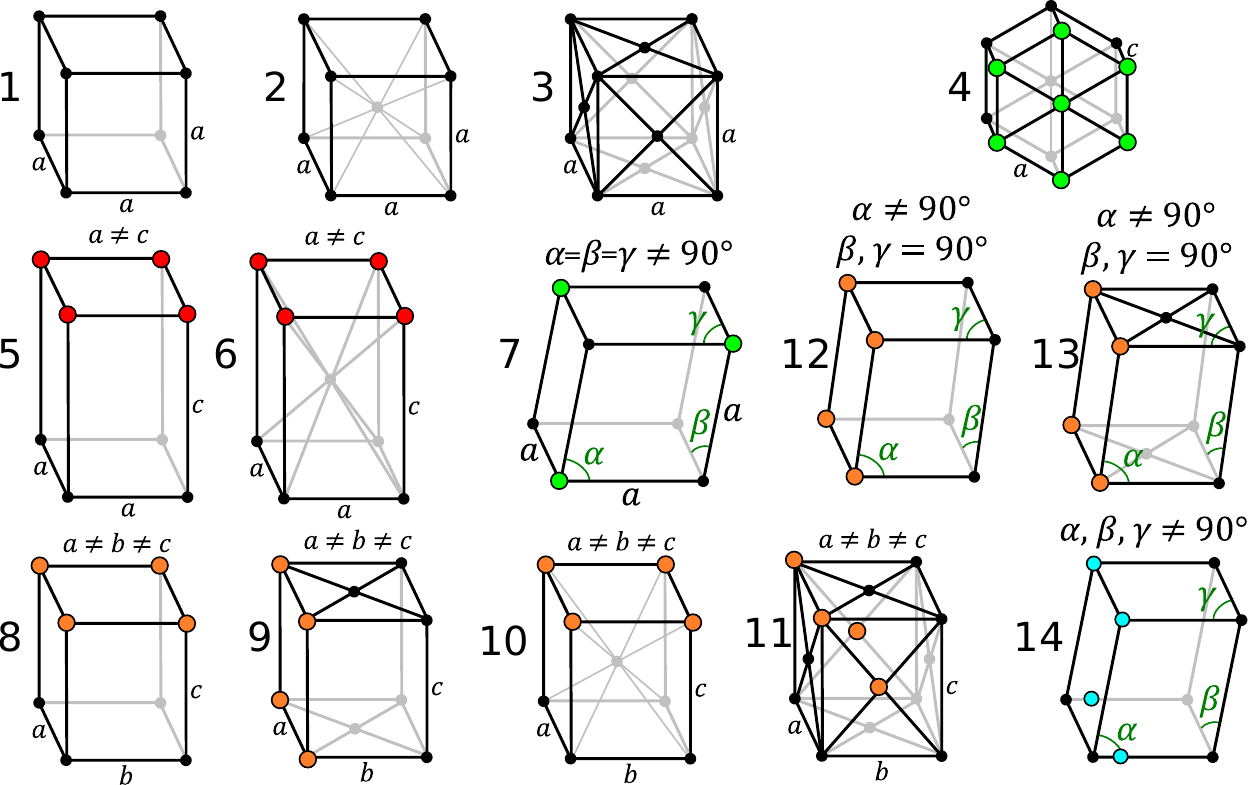}
\caption{Illustrations of the 14 Bravais types of lattices in $\R^3$.
The shaded nodes indicate how the lattices consist
of layers of two-dimensional lattices. Angles omitted 
in the figure are assumed to be $\pi/2$ (or $\pi/3$ in 4).
Edges labelled with equal letters are of equal length.
The image is taken from \cite{wik} and only slightly 
modified.\label{fig:3d}}
\end{figure}

\begin{proof}[Proof (of Theorem \ref{thm:satz2})]
We consider 6 cases (numbers 4-14 in Table \ref{tab1} and 
Figure \ref{fig:3d}, identified if they have equal point
groups). This will yield the entries of the last
column of Table \ref{tab1}, which shows the name of
the two-dimensional fundamental domain used, and the order
of the symmetry group of the corresponding three-dimensional 
fundamental domain (in parentheses).

{\bf Case 1:} Hexagonal (4). The lattice consists of equidistant
layers of hexagonal lattices. Attaching a thickened version
of the fundamental domain $F_{\triangle}$---say, $F:=F_{\triangle}
\times [-\ell/2,\ell/2]$, where $\ell$ is the distance between two 
adjacent layers---to each lattice point yields a tiling of $\R^3$. 
The symmetry group $S(F)$ of $F$ is $\ast12 \; 2$, the new
symmetries coming from rotating $F$ along an axis
which is parallel to the layers of hexagonal lattices about $\pi$ 
(``turning $F$ upside down''). 

{\bf Case 2:} Tetragonal (5-6). These lattices consist of equidistant 
layers of square lattices. Thus we can use the thickened version
of the fundamental domain $F_{\square}$ of the square lattice,
its symmetry being $\ast82$, having order 32.

{\bf Case 3:} Rhombohedral (7). This lattice consists of equidistant
layers of the hexagonal lattice $A_2$. So we can either use a 
thickened fundamental domain of $A_2$ with $D_6$-symmetry, yielding
a three-dimensional fundamental domain $F$ with $S(F)=\ast62$,
$|S(F)|=24$ and index $[S(F):P(\Gamma)]=2$. 
Or we can use a thickened version of $F_{\triangle}$ 
as in case 1, yielding a fundamental domain with $S(F)=\ast 12 \; 2$,
$|S(F)|=48$ and index $[S(F):P(\Gamma)]=4$. 

{\bf Case 4:} Orthorhombic (8-11). These lattices consist of equidistant 
layers of rectangular lattices. Thus we can use the thickened version
of the fundamental domain of the rectangular lattice,
its symmetry group being $\ast 42$ of order 16. The rectangular lattices are indicated 
in Figure \ref{fig:3d} by shaded points. 

{\bf Case 5:} Monoclinic (12-13). These two lattices also consist 
of equidistant layers of rectangular lattices. Thus we can reason 
as in the preceding case. 

Alternatively, we can use rectangular cuboids as fundamental
domains, having symmetry group $\ast 222$ of order 8.

{\bf Case 6:} Triclinic (14). This lattice---or rather: these
lattices---consist of layers of oblique lattices. We may use
a right prism over a parallelogram as a fundamental domain. It has
symmetry group $2\ast$ of order 4. Or we may even use  
a cuboid (erected on the rectangles of Figure \ref{fig:ob}).
This yields a fundamental domain with symmetry group  
$\ast222$ of order 8.
\end{proof}

\section{\bf Conclusions and Outlook} \label{sec:concl}


The above results motivate several further questions. We list below
a few that, to our knowledge, are completely open.

\subsection{Cubic and Rhombic lattices.} 
The reason that we excluded cubic lattices in Theorem \ref{thm:satz2}
is that there exist no fundamental domains $F$ for the cubic lattices 
$\Z^3$, bcc or fcc such that $S(F)$ contains $P(\Gamma)$ ($\Gamma \in \{ 
\Z^3, \mbox{bcc}, \mbox{fcc} \}$) as a proper subgroup of finite index. 
This is just because there are no such groups $S(F) \subset O(3)$. 
There still may be fundamental domains which have more symmetries in 
the sense that $|S(F)|>|P(\Gamma)|$ (but the authors doubt it).

The reason that we excluded rhombic lattices in Theorem \ref{thm:satz1} 
is that we were not able to find a general construction for fundamental 
domains $F$ for any rhombic lattice $\Gamma$ such that $S(F)$ is larger 
than $P(\Gamma)$. One can construct such fundamental domains for several 
particular cases, but a
general construction seems hard to obtain. If one tries
to use the same idea as in the other non-obtuse cases---start with 
a packing of polygons of higher symmetry (octagon, dodecagon, square) 
and refine---one runs into problems because the packings are
in general not vertex-to-vertex from some point on. 

\subsection{Even more symmetry.} 
Are there fundamental domains $F$ with $[S(F):P(G)] >2$? We have 
found a few: In the case of oblique plane lattices the
rectangular fundamental $F$ might be a square. (This happens if
$x=z$ in Figure \ref{fig:ob}.) In this case we obtain $[S(F):P(G)] =4$. 
In the case of the triclinic primitive lattice there is always a cuboidal
fundamental domain with $[S(F):P(G)] =4$. In analogy to the oblique lattices
in the plane, this cuboid might be a square prism or
even a cube in some particular cases. This would yield $[S(F):P(G)] =8$
or $[S(F):P(G)] =24$, respectively. What is the maximal 
value of the index $[S(F):P(G)]$ in $\R^d$ ($d \ge 2$)? Is the
maximal index always obtained by the lattices with the smallest
point group?

\subsection{Higher Dimensions.} 
The results in the present paper have been obtained by considering 
all different classes of lattices with respect to their 
symmetry group. There are 5 such classes in $\R^2$,
14 such classes in $\R^3$, 64 such classes in $\R^4$,
189 such classes in $\R^5$ and 826 such classes in $\R^6$
\cite{bbnwz,eng,ph,oeis,sou}. At some point it seems
desirable to find more general arguments than case-by-case
considerations. However, it is very likely that in higher 
dimensions there are several lattices $\Gamma$ with fundamental 
domains $F$ such that $P(\Gamma)$ is a proper subgroup of $S(F)$. 

\subsection{Non-Euclidean spaces.} 
The constructions used in this paper work also in spherical
or hyperbolic spaces, using spherical or hyperbolic regular 
$n$-gons. The fact that these $n$-gons are not similar to
each other on different length scales does not matter. All we
need is that there are edge-to-edge packings by regular $n$-gons 
on different length scales.

\subsection{Fractal Dimension.} 
The fundamental domains $F_{\square}$ of the square lattice, $F_{\triangle}$ 
of the hexagonal lattice and those of the rectangular lattices
with incommensurate basis lengths are of fractal appearance.
It might be possible to compute the Hausdorff dimensions of
the boundaries of these fundamental domains, as well as other fractal
dimensions, like the box-counting dimension or the affinity
dimension \cite{fal}, see also \cite{sing} and references therein. 
The two latter dimensions are particularly easy to compute
if one finds an iterated function system (IFS) generating the
fractal under consideration, see \cite{sing}. Up to the knowledge 
of the authors, no IFS for $F_{\square}$ or $F_{\triangle}$ or
the fundamental domains of rectangular lattices are known yet.

\subsection{Alternative Constructions.} 
The constructions used in this paper can be altered in many
ways. For instance, there are other ways to partition the octagons,
dodecagons and squares into two regions of different colours 
than the one used in the proof of Theorem \ref{thm:satz1}. 
All that is required is to keep the mirror symmetry of the 
partition, and take care that no overlaps occur. One possibility
is just to interchange the colours in the polygons of mixed colour.

\section*{Acknowledgements}
D. Frettl\"oh is grateful to two of the referees for several valuable 
remarks that improved the text a lot. The research 
leading to these results has received funding from the European
Research Council under the European Union's Seventh Framework 
Programme (FP7/2007-2013) / ERC grant agreement no 247029. 
J.R.C.G.~Damasco is grateful to the University of the Philippines System
for financial support through its Faculty, REPS, and Administrative
Staff Development Program.

\end{document}